\documentclass[11pt]{amsart}
\usepackage[left=1in,right=1in,top=1.2in,bottom=1in]{geometry}
\usepackage[utf8]{inputenc}
\usepackage{setspace}
\usepackage{amsmath,amssymb}
\usepackage{comment}
\usepackage{xcolor}
\usepackage{url}
\usepackage[colorlinks=true,allcolors=blue,backref=page]{hyperref}
\usepackage[noabbrev,capitalize,nameinlink]{cleveref}
\allowdisplaybreaks

\newtheorem{theorem}{Theorem}
\newtheorem{definition}[theorem]{Definition}

\newtheorem{proposition}[theorem]{Proposition}
\newtheorem{lemma}[theorem]{Lemma}

\newtheorem{corollary}[theorem]{Corollary}

\newtheorem{fact}[theorem]{Fact}

\newtheorem{prop}[theorem]{Proposition}

\newtheorem*{claim*}{Claim}

\theoremstyle{remark}
\newtheorem*{remark*}{Remark}
\newtheorem{remark}[theorem]{Remark}
\newtheorem{example}[theorem]{Example}

\numberwithin{theorem}{section}

\renewcommand{\phi}{\varphi}

\renewcommand{\leq}{\le}
\renewcommand{\geq}{\ge}

\newcommand{\eps}{\varepsilon}

\newcommand{\Ebar}{\overline{E}}

\renewcommand{\le}{\leqslant}
\renewcommand{\ge}{\geqslant}

\newcommand{\R}{\mathbb R}

\newcommand{\rand}{\operatorname{rand}}

\title
[]{The binomial random graph is a bad inducer}

\author{Vishesh Jain}
\author{Marcus Michelen}
\address{University of Illinois Chicago}
\email{visheshj@uic,edu, michelen@uic.edu}
\author{Fan Wei}
\address{Duke University}
\email{fan.wei@duke.edu}

\begin{document}
	\begin{abstract}
        For a finite graph $F$ and a value $p \in [0,1]$, let $I(F,p)$ denote the largest $y$ for which there is a sequence of graphs of edge density approaching $p$ so that the induced $F$-density of the sequence approaches $y$.   We show that for all $F$ on at least three vertices and all $p \in (0,1)$, the binomial random graph $G(n,p)$ has induced $F$-density strictly less than $I(F,p).$    
        This provides a negative answer to a problem posed by Liu, Mubayi and Reiher \cite{liu2023feasible}. 
        
        Our approach is in the limiting setting of graphons, and we in fact show a stronger result: the binomial random graph is never a \emph{local} maximum in the space of graphons of edge density $p$.  This is done by finding a sequence of balanced perturbations of arbitrarily small norm that increase the $F$-density.   
	\end{abstract}	
\maketitle

\section{Introduction}

For a finite labeled graph $G$ with vertex set $V(G)$ and edge set $E(G)$, recall that the edge density of $G$ is given by $\rho(G) = |E(G)|/ \binom{|V(G)|}{2}\,.$  Given another finite labeled graph $F$, let $$N(F,G) := \left|\{\phi : V(F) \hookrightarrow V(G) : (a,b) \in E(F) \iff (\phi(a),\phi(b)) \in E(G)\}   \right| $$ be the number of induced copies of $F$ in $G$ and define the induced $F$-density of $G$ to be $$\rho(F,G) := \frac{N(F,G)}{(|V(G)|)_{|V(F)|}}$$
where we write $(x)_k := x(x-1)\cdots(x-(k-1))$ for the falling factorial.  Finally, define the maximum induced $F$-density at edge density $p \in [0,1]$ via \begin{equation*}
    I(F,p) := \sup\left\{y : \exists~\{G_n\}_{n\geq 1}, \lim_{n\to\infty} |V(G_n)| = \infty, \lim_{n \to \infty} \rho(G_n) = p, \lim_{n \to \infty}   \rho(F,G_n) =y\right\}\,.
\end{equation*}
Informally, $I(F,p)$ is the largest induced $F$-density among large graphs of edge density approaching $p$. The maximum value of $I(F,p)$ over $p \in [0,1]$ is exactly the inducibility of $F$, introduced by Pippenger and Golumbic \cite{pippenger1975inducibility}.

Linearity of expectation shows that the expected induced $F$-density in the binomial random graph $G(n,p)$ is precisely $$\rand(F,p) := p^{|E(F)|}(1-p)^{\binom{n}{2} - |E(F)|}\,.$$ By basic concentration estimates, if we set $G_n$ to be an instance of $G(n,p)$ for each $n$, then we almost-surely have $\rho(G_n) \to p$ and $\rho(F,G_n) \to \rand(F,P)$.  As such, we always have $\rand(F,p) \leq I(F,p).$

 The question of whether random graphs are the extremal constructions for $I(F,p)$ seems to have been first explicitly investigated by
Even-Zohar and Linial \cite{even2015note}, who suggested exploring the performance of random constructions in maximizing the inducibility of $F$. In particular, they left open whether for $F$ given by the disjoint union of a path of length $3$ and an isolated vertex, the inducibility is achieved (in the limit) by $G(n,3/10)$. Perhaps suggesting that binomial random graphs can be optimal inducers in some examples, Liu, Mubayi and Reiher asked ``an easier question'' \cite[Problem~1.6]{liu2023feasible} whether there is a graph $F$ and $p \in (0,1)$ so that $I(F,p) = \rand(F,p)$\footnote{We note that Liu, Mubayi and Reiher work with unlabeled graphs rather than labeled graphs, but this only changes the quantities $N(F,G)$ by a factor depending only on $F$.}. In this paper, we provide a negative answer to this question. 
\begin{theorem}\label{thm:main}
    For each finite labeled graph $F$ with $|V(F)| \geq 3$ and for all $p \in (0,1)$, $$I(F,p) > \rand(F,p).$$ 
\end{theorem} 

We observe that if $|V(F)| \leq 2$ then for all $G$, $\rho(F,G)$ is a function solely of the edge density $\rho(G)$ and so the assumption of $|V(F)| \geq 3$ is required.

Understanding whether random constructions are close to optimal is one of the main themes in extremal combinatorics and related fields.  Perhaps the most closely related example is the Sidorenko conjecture \cite{Sidorenko, Sidorenko2}, which predicts that the number of (not-necessarily-induced) copies of a bipartite graph in a graph of edge density $p$ is \emph{minimized} in the binomial random graph.  Despite significant partial progress, Sidorenko's conjecture remains open in this full level of generality.  We note that in contrast to the Sidorenko conjecture, we show in \cref{thm:main} that the binomial random graph is \emph{never} the extremizer for inducibility.

In fact, we prove the stronger statement (\cref{thm:main-local}) that for each finite labeled graph $F$ with $|V(F)| \geq 3$ and for all $p \in (0,1)$, the constant graphon $W_p \equiv p$ is not a local maximizer of the function $W \mapsto \rho(F,W) := \rho_F(W)$, suitably defined for graphons. Before stating this formally, we introduce some terminology.

Consider a finite labeled graph $F$. Identify its vertex set $V(F)$ with $[m]$ and write $E$ for its edge set.  Let $\Ebar := \binom{[m]}{2} \setminus E$ be the set of non-edges of $F$.  Recall that a \emph{graphon} is a symmetric measurable function $W:[0,1]^2 \to [0,1]$.  For a graphon $W$, the edge density is given by
\begin{equation*}
    \rho(W) := \int_{[0,1]^2}W(x_1,x_2)\,dx_1\,dx_2
\end{equation*}
and the induced density of $F$ in $W$ is given by
\begin{equation} \label{eq:induced-density}
    \rho_F(W) := \int_{[0,1]^m} \prod_{e \in E} W(x_{e_1},x_{e_2}) \prod_{f \in \overline{E}}(1 - W(x_{f_1},x_{f_2})) \,dx_1 dx_2\cdots dx_m\,.
\end{equation}
We note that the constant graphon $W_p \equiv p$ is the limit of the random graphs $G(n,p)$ and that $\rho_F(W_p) = \rand(F,p)$. A \emph{kernel} is a bounded symmetric measurable function $W: [0,1]^2 \to \R$.  Of particular importance will be \emph{balanced} kernels, i.e.\ those that satisfy $\int_0^1 W(x,y)\,dy = 0$ for almost all $x \in [0,1]$. 

It follows from standard considerations (e.g.~\cite[Lemma~2.4]{lovasz2006limits}) that $I(F,p)$ can be recast as an optimization problem over graphons.
\begin{fact}
\label{fact:graphons}
    For every finite labeled graph $F$ and $p \in [0,1]$ we have $$I(F,p) = \sup_W\{ \rho_F(W) : \rho(W) = p\}\,.  $$
\end{fact}

The following result shows that, in the setting of \cref{thm:main}, $W_p$ is not even a local maximum. By \cref{fact:graphons}, it immediately implies \cref{thm:main}.

\begin{theorem}
\label{thm:main-local}
For each finite labeled graph $F$ with $|V(F)| \geq 3$, for all $p \in (0,1)$, and for all $\delta \leq \min\{p, 1-p\}$, there exists a kernel $\Delta = \Delta(F,p,\delta): [0,1]^2 \to [-\delta, \delta]$ such that
\begin{itemize}
    \item $\rho(\Delta) = 0$, and hence, $\rho(W_p + \Delta) = \rho(W_p) = p$,
    \item $\rho_F(W_p + \Delta) > \rho_F(W_p) = \rand(F,p)$.
\end{itemize}
\end{theorem}
Note that, by the restriction on the range of $\Delta$, $W_p + \Delta$ is a graphon.

On the other hand, we show (see \cref{sec:limit-perturbation}) that for $F = C_5$ and $p = 1/2$, for \emph{any} non-zero symmetric measurable function $\Delta: [0,1]^2 \to [-1,1]$ with $\rho(\Delta) = 0$, there exists $\delta_0 = \delta_0(\Delta) > 0$ such that for all $\delta \leq \delta_0$,
\[\rho(W_p + \delta \Delta) < \rho(W_p).\]
In words, while there are examples for which $W_p$ is a local maximum along any line, it is never a local maximum.  This obstacle is overcome in \cref{thm:main-local} by taking $\Delta$ to have non-linear dependence on $\delta$.

\subsection{Overview of the Proof}
The general strategy of the proof of \cref{thm:main-local} is to apply \cref{eq:induced-density} for a graphon of the form $W_p + \Delta$ and expand in terms of $\Delta$.  To start with, for a finite labeled graph $H$ we fix an orientation of its edges $E(H)$ arbitrarily and define 
$$t(H,\Delta) := \int_{[0,1]^{|V_H|}} \prod_{e \in E(H)}\Delta(x_{e_1},x_{e_2}) \,d\mathbf{x}\,. $$
Note that if $H_0$ and $H_1$ are isomorphic, then $t({H_0},\Delta) = t({H_1},\Delta)$ for all $\Delta$.  In the case when $\Delta$ is a graphon, the function $t(H,\Delta)$ counts the density of (not necessarily induced) copies of $H$ in $\Delta$.  Expanding \cref{eq:induced-density} in terms of $\Delta$ shows that
\begin{equation}\label{eq:expansion-intro}
    \rho_F(W_p + \Delta) = \rand(F,p) + \sum_{H} P_{H,F}(p) \cdot t(H,\Delta)
\end{equation}
where the sum is over non-empty subgraphs of $K_{v(F)}$ without isolated vertices and $P_{H,F}(p)$ are polynomials in $p$ depending on $H$ and $F$ (see \cref{lem:expansion}). 

The task is now clear: one wishes to choose $\Delta$ with $\rho(\Delta) = 0$ and $W_p + \Delta \in [0,1]$ so that the right-hand sum is positive.  A natural first idea in this direction is to choose $\Delta = \eps \Delta_0$ for some suitable kernel $\Delta_0$ and some small $\eps$.  A benefit of this is that we have the simple identity $t(H,\eps \Delta_0) = \eps^{e(H)} t(H,\Delta_0),$ and so taking $\eps$ small enough favors the terms in the sum indexed by graphs $H$ with few edges. Then,  one would simply like to make the leading order term positive. A major obstacle to this basic ``linear perturbative'' approach is that for certain graphs $H$ we have $t(H,\Delta) > 0$ whenever $\Delta \not\equiv 0$; a classical example is the case of $H = C_4$, although there are many such examples (for instance the wider class of \emph{norming} graphs, see \cite{lovaszBook}).  As such, if one applies $\Delta = \eps \Delta_0$ for balanced $\Delta_0$, one may find themselves in the unlucky situation where the dominant non-zero term in the expansion in \cref{eq:expansion-intro} is the term corresponding to $H = C_4$; since we have no control on the sign of $t(C_4,\Delta)$---it is necessarily positive---we are at the mercy of the term $P_{C_4,F}(p)$, which indeed can be negative.  We will show that this sequence of seeming-coincidences occurs in the case of $F = C_5$ and $p = 1/2$: all lower order terms cancel out for a balanced kernel, the only contributing graph with $4$ edges is $H = C_4$, and the coefficient $P_{C_4,C_5}(1/2)$ is negative.  This example is detailed in \cref{sec:limit-perturbation}.  Still, we show in \cref{sec:warm-up} that this basic linear perturbative approach can be made to work in the vast majority of cases, namely for all but at most three values of $p \in (0,1)$ that depend on $F$. 

As a starting point to see how to get around this obstacle, \cref{lem:expansion} will in fact show that if we take $m := v(F)$ then  $P_{K_m,F}(p) \neq 0$ for all $p \in (0,1)$, thereby providing a concrete non-zero term in the expansion. Ideally, we would like to choose some balanced $\Delta$ for which $t(K_m,\Delta)$ is large and of the correct sign.  Building off of an algebraic construction of  Jagger, {\v{S}}{\v{t}}ov{\'\i}{\v{c}}ek, and Thomason \cite{jagger1996multiplicities}, we construct a kernel  $U_m$ for which $t(K_m,U_m) + |t(G,U_m)| < 0$ for all non-cliques $G$ with minimum degree at least 2 (see \cref{prop:key}). Consequently, restricting the sum in \cref{eq:expansion-intro} to subgraphs $H$ with minimum degree at least 2 (which is easily done, as we will see later) for which $P_{H,F}(p)\neq 0$, and slightly rescaling $U_m$ to $U'_m$ if needed, the maximum absolute value of $t(H,U'_m)$ is attained uniquely at \emph{some} clique. Here, we crucially use that the set of connected subgraphs for which $P_{H,F}(p) \neq 0$ contains at least one clique (in particular $K_m$, as discussed above). 

At this point, there are two remaining obstacles. First, while we have guaranteed that the magnitude of $t(H,U'_m)$ in the expansion is maximized at a unique clique, this may not be the case for $P_{H,F}(p)\cdot t(H,U'_m)$. Second, even if the sum in the expansion were dominated by $P_{H,F}(p)\cdot t(H,U'_m)$ for a unique clique $H$, it might be the case that this term is negative. To get around the first obstacle, we use the ``tensor-power trick'': since $t(H,(U'_m)^{\otimes k}) = t(H,U'_m)^k$, we can work with a sufficiently high tensor power $(U'_m)^{\otimes k}$ of the special kernel $U'_m$ to ensure that the sum in the expansion is dominated by $P_{H,F}(p)\cdot t(H,(U'_m)^{\otimes k})$ for a unique clique $H$. For the second obstacle, we use that cliques are not norming, so that there exists a non-zero kernel $W$ such that $t(H,W) < 0$ for our unique clique $H$. By tensoring with this kernel $W$, we can fix the sign problem; our choice of sufficiently large $k$ guarantees that the sum is still dominated by the (now positive) term corresponding to $H$.

\subsection{Organization of the paper} In \cref{sec:warm-up} we formally introduce the perturbative approach outlined above.  We then prove a weaker version of \cref{thm:main} showing that for any $F$, there are at most three possible values of $p \in (0,1)$ for which $W_p$ could be a (local) maximizer of $\rho_F(W)$ (see \cref{prop:three-points}) using a linear perturbation. In fact, we show that there exists a universal ``direction'' $\Delta$ such that, outside of these three values of $p$, $W_p$ is not even a local maximizer along the line $W_p + t\Delta$.  As mentioned above, this approach fails for $F = C_5, p = 1/2$; we prove this in \cref{sec:limit-perturbation}.

Following this warm-up section, we set out to prove our main theorem \cref{thm:main-local}.  First, we take for granted our construction of the kernel $U_m$ that favors cliques (\cref{prop:key}) and prove \cref{thm:main-local} in \cref{sec:main-proof}.  Finally, we then prove \cref{prop:key} in \cref{sec:proof-key-prop} by building off of an algebraic construction of  Jagger, {\v{S}}{\v{t}}ov{\'\i}{\v{c}}ek, and Thomason \cite{jagger1996multiplicities}.

\section{The linear perturbative approach and its limitations}
\label{sec:warm-up}

In this section, as a warm-up, we prove the following weakening of \cref{thm:main-local}.

\begin{proposition}\label{prop:three-points}
   There exists a universal kernel $\Delta: [0,1]^2 \to [-2,2]$ for which the following holds. For each finite labeled graph $F$ with $|V(F)| \geq 3$, there are at most three points  $p_1,p_2,p_3 \in (0,1)$ so that  for all $p \in (0,1) \setminus \{p_1,p_2,p_3\}$ there is $\sigma \in \{-1,1\}$ so that  for all $0 \leq \delta \leq \delta_0 = \delta_0(F,p) > 0$, $$\rho_F(W_p+ \sigma \delta \Delta) > \rho_F(W_p) = \rand(F,p).$$ 
\end{proposition} 

The specific values of $p_1,p_2,p_3$ are given as the zeros of an explicit cubic polynomial whose coefficients depend on $F$. 

Note that, in the statement of \cref{prop:three-points}, we are able to use an appropriate scaling of a universal kernel to show that $W_p$ is not a local maximizer for $p \in (0,1)\setminus \{p_1,p_2,p_3\}$. The next proposition shows that there is no universal kernel which can be used to certify that $W_p$ is not a local maximizer for all $F$ with $|V(F)|\geq 3$ and all $p \in (0,1)$. 

\begin{prop}\label{fact:C5} (see \cref{sec:limit-perturbation})
    Let $p = 1/2$ and $F = C_5$.  For every kernel $\Delta$ with $\iint \Delta(x,y)\,dx\,dy = 0$ and $\Delta \not\equiv 0$ there are constants $c(\Delta) = c > 0$ and $\eps_0(\Delta) = \eps_0 > 0$ so that for all $|\eps| \leq \eps_0$ we have $$\rho_F(W_p + \eps \Delta) \leq \rand(F,p) - c \eps^4 \,.$$ 
\end{prop}

\subsection{The linear perturbative approach}

Let  $\Delta:[0,1]^2 \to \R$ be a kernel.
For a finite labeled graph $H$, fix an orientation of its edges $E(H)$ arbitrarily and define $$t(H,\Delta) = \int_{[0,1]^{|V_H|}} \prod_{e \in E(H)}\Delta(x_{e_1},x_{e_2}) \,d\mathbf{x}\,. $$
Note that if $H_0$ and $H_1$ are isomorphic, then $t({H_0},\Delta) = t({H_1},\Delta)$ for all $\Delta$.  In the case when $\Delta$ is a graphon, the function $t(H,\Delta)$ counts the density of (not necessarily induced) copies of $H$ in $\Delta$.

Throughout this section, we fix a finite labeled graph $F$. We identify its vertex set $V(F)$ with $[m]$ and write $E$ for the set of its edges and $\Ebar$ for the set of its non-edges.  
\begin{lemma}
\label{lem:expansion}
For a kernel $\Delta:[0,1]^2 \to \R$, we have the expansion $$\rho_F(W_p + \Delta) = \rand(F,p) +\sum_{H} P_{H,F}(p) \cdot  t(H,\Delta)\,,$$
where the sum is over non-empty unlabeled subgraphs of the complete graph $K_m$ without isolated vertices and $P_{H,F}(p)$ are polynomials in $p$ depending only on $H$ and $F$. 
Further, the polynomials $P_{H,F}(p)$ are given by $$P_{H,F}(p) = \frac{\rand(F,p)}{(p(1-p))^{e(H)}} \sum_{j = 0}^{e(H)} (1-p)^{e(H) - j} (-p)^j n_j(H,F) $$
where $n_j$ counts the number of sets $\{e_1,\ldots,e_{e(H) - j},f_1,\ldots,f_j\}$ with $e_i \in E(F)$, $f_i \in \overline{E(F)}$ so that the graph given by $\{e_1,\ldots,e_{e(H) - j},f_1,\ldots,f_j\}$ is isomorphic to $H$. In particular, $P_{K_m, F}(p) = (-1)^{|\overline{E(F)}|}$.
\end{lemma}
\begin{proof}
This follows from writing  
\begin{align*} \rho_F(W_p+\Delta) &= \int_{[0,1]^m} \prod_{e \in E} (p + \Delta(x_{e_1},x_{e_2})) \prod_{f \in \Ebar}(1 - p - \Delta(x_{f_1},x_{f_2})) \,d \mathbf{x} \\
&= \rand(F,p)  \int_{[0,1]^m} \prod_{e \in E} \left(1 + \frac{\Delta(x_{e_1},x_{e_2})}{p} \right) \prod_{f \in \Ebar}\left(1 - \frac{\Delta(x_{f_1},x_{f_2})}{1-p}\right) \,d \mathbf{x}
\end{align*}
expanding the products and collecting the terms.
\end{proof}

\begin{corollary}\label{cor:local-zero}
Let $\Delta$ be a balanced kernel, i.e., for almost all $x \in [0,1]$, we have \begin{equation}\label{eq:local-zero}
\int_{0}^1 \Delta(x,y) \,dy = 0\,. 
\end{equation}
Then $$ \rho_F(W_p + \eps \Delta) = \rand(F,p) + \eps^3 P_{K_3,F}(p) t(K_3,\Delta) + O(\eps^4)\,.$$
\end{corollary}
\begin{proof}
    Note that by \cref{eq:local-zero}, for any $H$ that has a vertex of degree $1$ we have $t(H,\Delta) = 0$.  The conclusion now follows by observing that the only graph without isolated vertices with at most $3$ edges and no vertices of degree $1$ is $K_3$.
\end{proof}

\begin{proof}[Proof of \cref{prop:three-points}]
    Consider the kernel $\Delta$ given by breaking $[0,1]^2$ into $9$ equal-sized squares of dimensions $1/3 \times 1/3$ and placing the following constant values on these squares:
    $$\begin{bmatrix} 
        2 & - 1& - 1 \\ -1 & 1  & 0 \\ -1 & 0 & 1
    \end{bmatrix}\,. $$
    Note that since the row sums of this matrix are $0$, we have that $\Delta$ satisfies the hypotheses of \cref{cor:local-zero}.  Further, a direct computation shows that $t(K_3,\Delta) = \frac{28}{27}$.   Thus \cref{cor:local-zero} implies  $$\rho_F(W_p + \eps\Delta) = \rand(F,p) + \eps^3 P_{K_3,F}(p) \cdot \frac{28}{27} + O(\eps^4)\,.$$

    If $P_{K_3,F}(p) \neq 0$, then we may take $\eps$ small enough and of the same sign as $P_{K_3,F}(p)$ to simultaneously ensure that both $W_p + \eps \Delta \in [0,1]$ and $\rho_F(W_p + \eps\Delta) > \rand(F,p)$.  Since $P_{K_3,F}(p)$ is a polynomial of degree $3$ in $p$ which is not identically $0$, there are at most three values $p_1,p_2,p_3 \in [0,1]$ for which we may not do the above.   
\end{proof}

\begin{example}\label{example}
Let $F$ be the path of length three together with an isolated vertex. Based on a flag algebra computation, Even-Zohar and Linial asked \cite{even2015note} whether $G(n,3/10)$ maximizes the inducibility of this graph at density $3/10$. Using \cref{prop:three-points}, we show that this is not the case. Indeed, one can compute that $$n_0(K_3,F) = 0,\quad n_1(K_3,F) = 2, \quad n_2(K_3,F) = 5, \quad n_3(K_3,F) = 3$$ and so 
    $$P_{K_3,F}(p) = \frac{\rand(F,p)}{(p(1-p))^3}\left(- 2(1-p)^2 p + 5 (1-p)p^2 - 3p^3 \right).$$  The roots of the polynomial in parentheses are $0, 2/5$ and $1/2$, implying that $2/5$ and $1/2$ are the only exceptional values for \cref{prop:three-points} in this case.  In particular, $I(F,3/10) > \rand(F,3/10)$.
\end{example}

\subsection{Limitations of the linear perturbative approach}
\label{sec:limit-perturbation}
To understand the failure of the linear perturbative argument for $F = C_5$ at $p = 1/2$, we begin by noting that $C_5$ is self-complementary; for any self-complementary graph $F$ and each $H$ we have $n_{j}(H,F) = n_{e(H)-j}(H,F)$, which implies that $P_{H,F}(1/2) = 0$ for any $H$ with an odd number of edges and any self-complementary graph $F$. 

In particular, for $F = C_5$ and for any kernel $\Delta$ with $\iint \Delta(x,y) \,dx\,dy = 0$, we note that the expansion in \cref{lem:expansion} yields $$\rho_{C_5}(W_{1/2} + \Delta) = \rand(C_5,1/2) + \eps^2 P_{P_2,C_5}(1/2) \cdot t({P_2},\Delta)  + O(\eps^3)$$
since the term corresponding to two disjoint edges integrates to $0$ due to $\iint \Delta(x,y) \,dx\,dy = 0$.

To compute $P_{P_2,C_5}(1/2)$ we note that $$n_0(P_2,C_5) = n_2(P_2,C_5) = 5,\qquad n_1(P_2,C_5) = 20$$ implying $$P_{P_2,C_5}(1/2) = \frac{1}{64}\left(5 \cdot \frac{1}{4} + 5 \cdot \frac{1}{4} - 20 \cdot \frac{1}{4} \right) = - \frac{5}{128}.$$
This implies that
$$\rho_{C_5}(W_{1/2} + \Delta) = \rand(C_5,1/2) - \eps^2 \cdot \frac{5}{32}\cdot t({P_2},\Delta) + O(\eps^3)\,.$$

Note that \begin{equation}\label{eq:P_2-count}
t(P_2,\Delta) = \int_0^1\int_0^1\int_{0}^1 \Delta(x,y)\Delta(x,z)\,dz\,dy\,dx = \int_0^1 \left( \int_0^1 \Delta(x,y)\,dy\right)^2\,dx\,.
\end{equation}

In particular, we have $t(P_2,\Delta) \geq 0$.  Thus, if $t(P_2,\Delta) \neq 0$ we have $$\rho_{C_5}(W_{1/2} + \Delta) = \rand(C_5,1/2) - O(\eps^2)$$ completing the proof of \cref{fact:C5} in the case of $t(P_2,\Delta) \neq 0$.  If we have $t(P_2,\Delta) = 0$ then by \cref{eq:P_2-count} we have $$\int_0^1 \Delta(x,y)\,dy = 0$$ for almost all $x$.  In particular, for any graph $H$ with a vertex of degree $1$ we have that $t(H,\Delta) = 0$.  Recall that since $C_5$ is self-complementary, in the expansion of $\rho_{C_5}(W_{1/2} + \eps \Delta)$, all graphs $H$ with an odd number of edges vanish.  The only graphs with at most four edges that have no vertices of degree $1$ are $C_3$ and $C_4$; recalling that $P_{C_3,C_5}(1/2) = 0$ due to $C_5$ being self-complementary, we see that in the case when $t(P_2,\Delta) = 0$ the expansion from \cref{lem:expansion} becomes \begin{equation}\label{eq:fourth-expansion}
\rho_{C_5}(W_{1/2} + \eps \Delta) = \rand(C_5,1/2) + \eps^4 P_{C_4,C_5}(1/2) t(C_4,\Delta) + O(\eps^5)\,. \end{equation}

We directly compute $$n_0(C_4,C_5) = n_4(C_4,C_5) = 0\,,\quad n_1(C_4,C_5) = n_3(C_4,C_5) = 5\,,\quad n_2(C_4,C_5) = 0$$
and so $$P_{C_4,C_5}(1/2) =  -\frac{5}{32}\,.$$

Importantly, the functional $\Delta \mapsto t(C_4,\Delta)^{1/4}$ is a norm on the space of kernels (see e.g. \cite[Prop 14.2]{lovaszBook}) and so $t(C_4,\Delta) = 0 \iff \Delta \equiv 0.$  Since we have assumed $\Delta \not\equiv 0$, we have that $t(C_4,\Delta) > 0$.  Plugging this into \cref{eq:fourth-expansion} completes the proof of \cref{fact:C5}.

\section{Proof of \cref{thm:main-local}}
\label{sec:main-proof}

Finally, we prove \cref{thm:main-local}. The construction of our kernels is quite a bit more involved compared to the simple, explicit proof of \cref{prop:three-points}. In particular, we will make extensive use of the tensor product operation, which we recall for the reader's convenience.  The use of the tensor product introduces the non-linearity required to upgrade \cref{prop:three-points} to \cref{thm:main-local} given the obstacle observed in  \cref{fact:C5}.

\begin{definition}
    \label{def:tensor-product}
  Fix any measure preserving map $\varphi : [0,1] \to [0,1]^2$. We write $\varphi(x) = (\varphi(x)_1, \varphi(x)_2)$. Given kernels $U, W : [0,1] \times [0,1] \to \R$, we define the tensor product $U\otimes W$ by
  \[(U \otimes W)(x,y) = U(\varphi(x)_1, \varphi(y)_1)W(\varphi(x)_2, \varphi(y)_2).\]
\end{definition}
\begin{remark}
    The choice of the measure preserving map is not important for our purpose, so we fix one arbitrarily; see the discussion in \cite[Section~7.4]{lovaszBook}.
\end{remark}

The key property of tensor products that we will need is the following.

\begin{fact}[see, e.g.,~{\cite[Equation~(7.17)]{lovaszBook}}]
    \label{fact:tensor-product-hom}
    For every finite labeled graph $F$ and any kernels $U$ and $W$,
    \[t(F,U\otimes W) = t(F,U)t(F,W). \]
\end{fact}

We will also need the existence of a balanced kernel $\Delta$ such that $t(K_z,\Delta) \neq 0$ for all $z \geq 3$. 

\begin{lemma}
\label{lem:balanced-kernel}    
There exists a kernel $B: [0,1]^2 \to [-1,1]$ such that $\int_0^1B(x,y)dy = 0$ for almost all $x \in [0,1]$ and $t(B,K_z)\neq 0$ for all $z \geq 3$. 
\end{lemma}
\begin{proof}
    Let $f:[0,1] \to [-1,1]$ be any function for which $\int_{0}^1 f(x)dx = 0 $ but $\int_0^1 f(x)^k dx \neq 0$ for all $k \geq 2$; for instance, we can take $f(x) = 1$ for $x \in [0,1/3]$ and $f(x) = -1/2$ for $x \in (1/3, 1]$. Let $B(x,y) = f(x)f(y)$. Then, for all $x \in [0,1]$, $\int_0^1 B(x,y)dy = f(x)\int_0^1 f(y)dy = 0$. Moreover, for all $z \geq 3$,
    \[t(B,K_z) = \left(\int_0^1 f(x)^{z-1}\right)^{z} \neq 0. \qedhere\]
\end{proof}

The main ingredient in our proof is the following proposition, whose proof we defer to \cref{sec:proof-key-prop}.

\begin{proposition}
    \label{prop:key} 
    For every $z \geq 3$, there exists a kernel $U_z : [0,1]\times [0,1] \to [-1,1]$ such that $t(K_z,U_z) + |t(G, U_z)| < 0$, simultaneously for all non-clique graphs $G$ with minimum degree at least $2$. 
\end{proposition}

Before proving \cref{prop:key}, let us show how it can be used to prove \cref{thm:main-local}. 

\begin{proof}[Proof of \cref{thm:main-local}]
    Let $F$ be a finite labelled graph with $m\geq 3$ vertices. Our perturbation $\Delta$ will be of the form
    \[\Delta = \delta\cdot B \otimes (\lambda \cdot U_m)^{\otimes N} \otimes W,\]
    where $\delta \leq \min(p,1-p)$ is the parameter in the statement of \cref{thm:main-local}, $B: [0,1]^2 \to [-1,1]$ is the kernel from \cref{lem:balanced-kernel}, $U_m$ is the kernel from \cref{prop:key}, $N \in \mathbb{Z}_+$ and $\lambda \in (0,1]$ will be chosen later, and $W$, which will also be chosen later, is either $U_z$ for some $3 \leq z \leq m$ from \cref{prop:key} or the constant $1$ graphon. Note that for any such choice of $N$ and $W$, $B\otimes U_m^{\otimes N} \otimes W : [0,1]^2 \to [-1,1]$, so that $\Delta : [0,1]^2 \to [-\delta, \delta]$.

    Moreover, since $B$ is balanced, $t(G,B) = 0$ for any graph with a vertex of degree $1$. Hence, $t(G,\Delta) = 0$ for any such graph $G$. In particular, by \cref{fact:tensor-product-hom}, $\rho(\Delta) = 0$, as required.

    We now show that for suitable $N$ and $W$, $\rho_F(W_p + \Delta) > \rho_F(W_p)$. Substituting the form of $\Delta$ in \cref{lem:expansion} and using \cref{fact:tensor-product-hom}, we have that
    \begin{align*}
        \rho_F(W_p + \Delta) &= \rand(F,p) + \sum_{H \in \mathcal{H}} P_{H,F}(p)\cdot t(H,\Delta)\\
        &= \rand(F,p) + \sum_{H \in \mathcal{H}} P_{H,F}(p)\cdot \delta^{|E(H)|}t(H,B)t(H,\lambda U_m)^N t(H,W)\\
        &= \rand(F,p) + \sum_{H \in \mathcal{H}_1} P_{H,F}(p)\cdot \delta^{|E(H)|}t(H,B)t(H,\lambda U_m)^N t(H,W), 
    \end{align*}
where $\mathcal{H}$ denotes all non-empty unlabeled subgraphs of $K_m$ without isolated vertices and $\mathcal{H}_1$ denotes those graphs in $\mathcal{H}$ which have minimum degree at least $2$; here, we used that for all $H \in \mathcal{H}\setminus \mathcal{H}_1$, $t(H,B) = 0$.

Let $\mathcal{K} = \mathcal{K}_{F,p} \subseteq \mathcal{H}$ denote the set of cliques $K_z$ for which $P_{K_z, F}(p)t(H,B) \neq 0$. Note that $K_m \in \mathcal{K}$ since $|P_{K_m,F}(p)| = 1$ (\cref{lem:expansion}) and $t(K_m,B) \neq 0$ by construction. Let $\mathcal{K}^* \subseteq \mathcal{K}$ denote the set of cliques $K_z \in \mathcal{K}$ for which
\[|t(K_z, U_m)| = |\max_{H \in \mathcal{H}_1}t(H,U_m)|;\]
note that by \cref{prop:key}, all the maximizers of the right-hand side must be cliques. By taking $\lambda = 1$ if $\mathcal{K}^*$ is a singleton, and $\lambda \in (0,1)$ sufficiently close to $1$ otherwise, we may ensure that there is a unique integer $z \in [3,m]$ and $\gamma \in (0,1)$ such that 
\[ |t(H,\lambda U_m)| \leq \gamma |t(K_{z}, \lambda U_m)| \]
for all $H \in \mathcal{H}_1 \setminus K_{z}$. In particular, $|t(K_z, \lambda U_m)| \geq |t(K_m, \lambda U_m)| > 0$. 

With this choice of $z$, we have
\begin{align*}
    \rho_F(W_p + \Delta) - \rand(F,p) &= c_{K_{z},F}(p) t(K_{z}, \lambda U_m)^N t(K_{z},W) + \sum_{H\in \mathcal{H}_1 \setminus K_{z}}c_{H,F}(p) t(H,\lambda U_m)^N t(H,W) \\
    &\geq c_{K_{z},F}(p) t(K_{z}, \lambda U_m)^N t(K_{z},W) - \sum_{H\in \mathcal{H}_1 \setminus K_{z}}|c_{H,F}(p) t(H,\lambda U_m)^N| \\
    &\geq c_{K_{z},F}(p) t(K_{z}, \lambda U_m)^N t(K_{z},W) - \gamma^N \sum_{H\in \mathcal{H}_1 \setminus K_{z}}|c_{H,F}(p) t(K_{z},\lambda U_m)^N|
\end{align*}
We wish to show that there is a choice of $N$ and $W$ for which the right hand side is positive. Note that, by construction, $c_{K_z,F}(p) \neq 0$. If it is positive, we take $W$ to be the constant one graphon; if it is negative, we take $W = U_z$ (from \cref{prop:key}): in either case, we have $c_{K_z, F}(p)t(K_z,W) > 0$. Now, taking $N$ to be a sufficiently large even positive integer and using that $|t(K_z, \lambda U_m)| > 0$ finishes the proof.
\end{proof}

\section{Proof of \cref{prop:key}}
\label{sec:proof-key-prop}
In this section, we prove \cref{prop:key}. First, we claim that it suffices to consider $z$ odd, $z \geq 5$. Indeed, for even $z$, the statement of the lemma was proved by Jagger, {\v{S}}{\v{t}}ov{\'\i}{\v{c}}ek, and Thomason \cite{jagger1996multiplicities} (see the next paragraph), and indeed, our construction draws inspiration from theirs, whereas for $z = 3$, we may simply consider the constant kernel $U_{3} \equiv -\alpha$ for any $\alpha \in (0,1)$, for which we have
\[t(K_3, U_3) + |t(G, U_3)| \leq -\alpha^3 + \alpha^4 < 0,\]
since any non-clique graph $G$ with minimum degree at least $2$ has at least $4$ edges.  

The work of Jagger, {\v{S}}{\v{t}}ov{\'\i}{\v{c}}ek, and Thomason \cite{jagger1996multiplicities} disproves a conjecture of Burr and Rosta \cite{BR}, and Erd\H{o}s \cite{Erdos} on Ramsey multiplicities by showing that any graph containing $K_4$ is not a common graph (roughly speaking, we say that $H$ is a common graph if among red/blue edge colorings of the complete graph, the monochromatic density of $H$ is asymptotically minimized by the uniformly random red/blue coloring). In the course of their proof, they prove the even $z$ case of \cref{prop:key}. To describe their kernel $U:[0,1]\times [0,1]\to [-1,1]$, we decompose $[0,1]$ into $2^k$ equally-sized intervals which we arbitrarily identify with $V:= \mathbb{F}_2^k$. We then define $U:V\times V \to [-1,1]$ by $U(x,y) := (-1)^{q(x+y)}$ for an explicit quadratic form $q$ on $V$. Using our notation, they show (see last equation in the proof of Lemma 10 and statement of Lemma 11 in \cite{jagger1996multiplicities}) that $t(H,U) = -2^{-k}$ if $H$ is a clique and $|t(H,U)| \leq 2^{-2k}$ if $H$ is non-empty and not a clique. Since their construction heavily relies on working in characteristic $2$, it is only able to prove \cref{prop:key} for even cliques. 

Our proof below uses several of the same high-level elements as \cite{jagger1996multiplicities}. However, since we are no longer working in characteristic $2$, there are considerably more number-theoretic intricacies. 

\begin{proof}[Proof of \cref{prop:key} for $z$ odd, $z\geq 5$]
Let $p$ be an odd prime such that $p$ divides $z-2$. Let $k$ be an integer parameter which will be chosen later. To define our kernel $U = U_z : [0,1] \times [0,1] \to [-1,1]$, we decompose $[0,1]$ into $p^{k}$ equally-sized intervals, which we (arbitrarily) identify with $V:=\mathbb{F}_p^k$. 

First, let $s$ be some fixed quadratic nonresidue modulo $p$ and define 
the quadratic form $q : V \to \mathbb{F}_p$ via 

\[q(x) := s \cdot x^T \mathrm{Id}_V  x = s(x_1^2 + \dots + x_k^2)\]
where we write $x = (x_1,\ldots,x_k)\,.$ 

We then define 
$U : V \times V \to [-1,1]$ by
\begin{align*}
    U(x,y) := \cos\left(2\pi \frac{q(x+y)}{p} \right) =  \frac{1}{2}\left(\exp(2\pi i q(x+y)/p) + \exp(-2\pi i q(x+y)/p) \right)\,.
\end{align*}

We begin by finding a more convenient expression for homomorphism densities into $U$. Let $G$ be a graph with $\ell$ vertices and edge set $E(G)$.  For an edge $e$ we let $e_1$ and $e_2$ be its two vertices ordered arbitrarily.
Let $E_{e}$ denote the $\ell \times \ell$ matrix (over $\mathbb{F}_p$) with the $e_1 e_1, e_1 e_2, e_2 e_1, e_2 e_2$  entries equal to one and all other entries equal to $0$. For $\sigma \in \{\pm 1\}^{E(G)}$, let $M_G^{\sigma} := \sum_{e \in E(G)} E_{e} \sigma_{e}$. 

Since $M = M_G^\sigma$ is a symmetric matrix over a field of characteristic different from $2$, it follows (see \cite[Chapter~1]{lam2005introduction}) that $C^T M C = D$ for some invertible matrix $C$ over $\mathbb{F}_p$ and diagonal matrix $D = D_G^\sigma =  \operatorname{diag}(d_1^{\sigma},\dots, d_\mathrm{rk}^\sigma,0,\dots,0)$, where $\mathrm{rk} = \mathrm{rk}(M)$ denotes the rank of $M$ over $\mathbb{F}_p$ (so $d_i^\sigma \neq 0$); note that $\mathrm{rk}(M) \geq 1$.  We note that the choice of $D$ need not be unique---it need not be the case that $C^T = C^{-1}$ and indeed $D$ need not be the eigenvalues of $M$---and  a choice in a particular case will be made later in \cref{lem:rank}. 

\begin{lemma}\label{lem:hom-U-expansion}
    In the notation above, we have $$2^{|E(G)|} t(G,U) = \sum_{\sigma \in \{\pm 1\}^{E(G)}} \prod_{j = 1}^{\mathrm{rk}(M_G^\sigma)} \left(\frac{g(d_j^\sigma s ; p)}{p} \right)^k $$
    where $g(\mu;p) = \sum_{x \in \mathbb{F}_p} \exp(2 \pi i \mu x^2 / p)$
    denotes the quadratic Gauss sum. 
\end{lemma}
\begin{proof} 
First expand
    \begin{align*}
        2^{|E(G)|}|V|^\ell t(G,U) &= \sum_{x_1,\dots,x_\ell \in V}  \prod_{e \in E(G)}\left(\exp(2\pi i q(x_{e_1} +  x_{e_2})/p) + \exp(-2\pi i q(x_{e_1} + x_{e_2})/p)\right) \\
        &= \sum_{\sigma \in \{\pm 1\}^{E(G)}} \sum_{x_1,\dots,x_\ell \in V} \exp \left(\frac{2\pi i}{p} \sum_{e \in E(G)} q(x_{e_1} + x_{e_2})\sigma_{e}\right).
    \end{align*}
    
    We identify the tensor product $\mathbb{F}_p^{\ell} \otimes V$ with $V^{\ell}$ in the natural manner (so, e.g., $(1,0,\dots,0) \otimes x \mapsto (x,0,\dots,0)$). We may thus rewrite  
    \begin{align*}
        \sum_{x_1,\dots,x_\ell \in V} \exp \left(\frac{2\pi i}{p} \sum_{e \in E(G)} q(x_{e_1} + x_{e_2})\sigma_{e}\right) &= \sum_{\mathbf{x} \in V^\ell} \exp \left(\frac{2\pi i}{p} \sum_{e \in E(G)} s (x_{e_1} + x_{e_2})^T (x_{e_1} + x_{e_2})\sigma_{e}\right) \\
        &= \sum_{\mathbf{x} \in V^\ell} \exp \left(\frac{2\pi i}{p} \sum_{e \in E(G)} s\cdot  \mathbf{x}^T E_{e} \otimes \mathrm{Id}_V \mathbf{x}\sigma_{e}\right) \\
        &= \sum_{\mathbf{x} \in V^\ell} \exp \left(\frac{2\pi i}{p}   \mathbf{x}^T \left(\sum_{e \in E(G)}  (E_{e}\sigma_{e}) \otimes s\mathrm{Id}_V \right) \mathbf{x}\right) \\
        &=  \sum_{\mathbf{x} \in V^{\ell}}\exp\left(\frac{2\pi i}{p} \mathbf{x}^T (M_G^\sigma \otimes s\mathrm{Id}_V) \mathbf{x}\right)\,.
    \end{align*}

    Using the invertibility of $C\otimes \operatorname{Id}_V$ on $V^{\ell}$, we may apply the invertible transformation $\mathbf{x} \mapsto (C \otimes \mathrm{Id}_V) \mathbf{x}$ to see
    \begin{align*}
    \sum_{\mathbf{x} \in V^{\ell}}\exp\left(\frac{2\pi i}{p} \mathbf{x}^T (M_G^\sigma \otimes s\mathrm{Id}_V) \mathbf{x}\right)
    &= \sum_{\mathbf{x} \in V^{\ell}}\exp\left(\frac{2\pi i}{p} \mathbf{x}^T (C \otimes \operatorname{Id}_V) ^T(M_G^\sigma \otimes s\mathrm{Id}_V)(C \otimes \operatorname{Id}_V) \mathbf{x}\right) \\
    &= \sum_{\mathbf{x} \in V^{\ell}}\exp\left(\frac{2\pi i}{p} \mathbf{x}^T (D_G^\sigma \otimes s\mathrm{Id}_V)\mathbf{x}\right)\\
    &= \sum_{x_1,\dots, x_{\ell} \in V}  \prod_{j=1}^{\mathrm{rk}(M)}\exp\left(\frac{2\pi i d_j^\sigma s}{p} x_j^T  x_j\right)\\
    &= |V|^{\ell - \mathrm{rk}(M)}\prod_{j=1}^{\mathrm{rk}(M)} \sum_{x  \in V} \exp\left(\frac{2\pi i d_j^\sigma s x^T x }{p} \right)\\
    &= |V|^{\ell}\prod_{j=1}^{\mathrm{rk}(M)}(g(d_j^\sigma s; p)/p)^{k},
    \end{align*}
    where in the last line we have used $|V| = p^k$.
\end{proof}

To simplify this expression, we recall the values of quadratic Gauss sums. 

\begin{fact}[see,~e.g.~{\cite[Section~53]{nagell2021introduction}}]
    \label{fact:gauss-sum}
    For any $\mu \in \mathbb{F}_p^\times$, 
    \[g(\mu; p) = \left(\frac{\mu}{p}\right)g(1;p),\]
    where $(\frac{\cdot}{\cdot})$ denotes the Legendre symbol. Moreover,
    \begin{align*}
        g(1;p) = \begin{cases}
            \sqrt{p} \quad \text{if } p\equiv 1 \mod 4, \\
            i\sqrt{p} \quad \text{if } p\equiv 3 \mod 4.
        \end{cases}
    \end{align*}
\end{fact}

\begin{corollary}\label{cor:hom-rank-expansion}
    In the above notation we have \begin{equation}
    \label{eqn:homomorphism-rank-expansion}
        \left|2^{|E(G)|}t(G,U) - \sum_{\sigma \in \Sigma_1} g(d_1^\sigma s;p)/p)^k \right| \leq  2^{|E(G)|}p^{-k},
    \end{equation}
    where $\Sigma_1$ denotes the set of $\sigma \in \{\pm 1\}^{E(G)}$ for which $\mathrm{rk}(M_G^\sigma) = 1$.
\end{corollary}
\begin{proof}
    Since $p$ is odd, $|g(\mu;p)| = \sqrt{p}$ for all $\mu \in \mathbb{F}_p^\times$.  Applying this bound to \cref{lem:hom-U-expansion} completes the proof. 
\end{proof}

In order to proceed, we need the following linear algebraic lemma. This is the only place where we use the assumption that $p$ divides $z-2$. The upshot of this lemma is that the only graphs for which $\Sigma_1$ in \cref{eqn:homomorphism-rank-expansion} is non-empty are cliques; moreover, in this case, as we send $k \to \infty$, the right hand side of \cref{eqn:homomorphism-rank-expansion} is negligible compared to the contribution from $\Sigma_1$. 

\begin{lemma}
    \label{lem:rank}
    Let $G$ be a graph with minimum degree at least two. The following hold:
    \begin{enumerate}
        \item If $G$ is not a clique, then $\mathrm{rk}(M_G^\sigma) \geq 2$ for all $\sigma \in \{\pm 1\}^{E(G)}$. 

        \item If $G = K_z$ and $\sigma = (1,\dots,1)$ or $\sigma = (-1,\dots,-1)$, then $\mathrm{rk}(M_G^\sigma) = 1$. 

        \item If $G = K_z$ and $\mathrm{rk}(M_G^\sigma) = 1$, then we may take $d_1^\sigma \in \{\pm 1\}$. 
    \end{enumerate}
\end{lemma}

\begin{proof}
For (2), it suffices to consider the case $\sigma = (1,\dots,1)$ since $M_G^{-\sigma} = -M_G^\sigma$. In this case, $M_G^\sigma$ is a $z\times z$ matrix with off-diagonal entries equal to $1$ and diagonal entries equal to $z-1 \equiv 1 \mod p$ (by our assumption that $p$ divides $z-2$), i.e., $M_G^\sigma$ is the constant all ones matrix, which has rank $1$. 

For (1) and (3), first observe that $\mathrm{rk}(M_G^\sigma) \geq 1$. Suppose $\mathrm{rk}(M_G^\sigma) = 1$.  Since $G$ has no isolated vertices, every row and every column of $M_G^\sigma$ must have an off-diagonal entry equal to $\pm 1$. Moreover, since $\mathrm{rk}(M_G^\sigma) = 1$, all rows (columns) must be multiples of the first row (column). This shows that all entries of $M_G^\sigma$ must be $\pm 1$. (1) follows immediately, since for $M_G^\sigma$ to have all off-diagonal entries non-zero, $G$ must necessarily be a clique.

For (3), let $G = K_z$ and suppose $\mathrm{rk}(M_G^\sigma) = 1$. As discussed above, all entries of $M_G^\sigma$ must be $\pm 1$ and all rows of $M_G^\sigma$ are either equal to the first row or the negative of the first row. Let $v = (v_1,\dots, v_z)$, where $v_i \in \{\pm 1\}$, denote the first row of $M_G^\sigma$. Let $\gamma_i = 1$ (respectively, $\gamma_i = -1$) if the $i^{th}$ row of $M_G^\sigma$ equals $v$ (respectively, $-v$). Note that, since $M_G^\sigma$ is symmetric and all rows of $M_G^\sigma$ are either equal to $v$ or $-v$, it follows that $v_i = v_1$ if $\gamma_i = 1$ and $v_i = -v_1$ if $\gamma_i = -1$, i.e.,~$\gamma_i = v_i/v_1$. Therefore,
\[(M_G^\sigma)_{ij} = \gamma_i v_j = (v_i v_j)/v_1 = v_1 (vv^T)_{ij},\]
so that $M_G^\sigma = v_1 \cdot vv^T$. 

Now, let $C$ be an invertible $\ell \times \ell$ matrix over $\mathbb{F}_p$ whose first row is $v$. We then note that $$M = C^T \mathrm{diag}(v_1,0,0,\ldots,0) C\,.$$
Recalling that $v_1 \in \{\pm 1\}$ completes the proof.
\end{proof}

By combining \cref{eqn:homomorphism-rank-expansion} and \cref{lem:rank}(1), we have that for any non-clique $G$ with minimum degree at least $2$, 
\begin{equation}
    \label{eqn:non-clique}
    |t(G,U_k)| \leq  p^{-k}
\end{equation}
for all $k$, where we have included the subscript $U_k$ to emphasize the dependence on $k$.  

For the case of $G = K_z$, $z$ odd, $z\geq 5$, we consider two cases depending on whether $p \equiv 1 \mod 4$ or $p \equiv 3 \mod 4$. 

\paragraph{\bf Case 1: $p \equiv 1 \mod 4$} Since $s$ is a quadratic nonresidue modulo $p$, so is $-s$. Therefore, by \cref{fact:gauss-sum}, $g(s;p) = g(-s;p) = -g(1;p) = -\sqrt{p}$. Hence, for any odd $k$ which is sufficiently large, it follows from \eqref{eqn:homomorphism-rank-expansion} that
\begin{align}
     t(K_z, U_k) &\leq p^{-k} \left( |\Sigma_1| 2^{-\binom{z}{2}}(-\sqrt{p})^{k}
 + 1\right) \nonumber \\
    \label{eqn:1mod4}
    &\leq  p^{-k} \left( 2\cdot 2^{-\binom{z}{2}}(-\sqrt{p})^{k}
 + 1\right)\nonumber\\
    &\leq - 2^{-\binom{z}{2}} p^{-k/2}
\end{align}
where the second inequality  follows from \cref{lem:rank}(2) since $k$ is odd and the final inequality follows from taking $k$ is sufficiently large as a function of $p$ and $z$

\paragraph{\bf Case 2: $p\equiv 3 \mod 4$} Since $s$ is a quadratic nonresidue modulo $p$, $-s$ is a quadratic residue. Therefore, by \cref{fact:gauss-sum}, $g(s;p) = -g(-s;p) = -g(1;p) = -i\sqrt{p}$. Hence, for any $k \equiv 2 \mod 4$ which is sufficiently large (depending on $z$), we have as before that
\begin{align}
    \label{eqn:3mod4}
     t(K_z, U_k) &\leq p^{-k} \left( -|\Sigma_1| 2^{-\binom{z}{2}}(\sqrt{p})^{k}
 + 1\right)\nonumber \\
    &\leq p^{-k} \left( -2\cdot 2^{-\binom{z}{2}}(\sqrt{p})^{k}
 + 1\right) \nonumber\\
    &\leq - 2^{-\binom{z}{2}} p^{-k/2};
\end{align}

In either case, we see that there is an infinite sequence of integers $k$ such that simultaneously, for all non-clique $G$ with minimum degree at least $2$,
\[t(K_z, U_k) + |t(G,U)| \leq -2^{-\binom{z}{2}}p^{-k/2} + p^{-k};\]
choosing $k$ sufficiently large with respect to $z$ completes the proof. 
\end{proof}

\section*{Acknowledgments}
The authors thank Emily Cairncross and Dhruv Mubayi for introducing this problem to us and Jacob Fox and Huy Tuan Pham for useful discussions.  V.J.\ is supported in part by NSF grant DMS-2237646 and M.M.\ is supported in part by NSF grants DMS-2137623 and DMS-2246624. F.W.\ is supported by NSF grants DMS-2404167 and DMS-2401414.

\bibliographystyle{abbrv}
\bibliography{main.bib}

\end{document}